\documentclass[twoside]{amsart}
\pdfoutput=1
\usepackage{amsmath,amssymb}
\usepackage{stmaryrd}
\usepackage[mathscr]{eucal}

\newtheorem{proposition}{Proposition}[section]
\newtheorem{definition}[proposition]{Definition}
\newtheorem{lemma}[proposition]{Lemma}
\newtheorem{corollary}[proposition]{Corollary}
\newtheorem{example}[proposition]{Example}

\numberwithin{equation}{section}

\newcommand{\etal}{\textit{et al.}}
\newcommand{\rank}{\operatorname{rank}}
\newcommand{\submax}{\operatorname{submax}}
\newcommand{\A}{\mathbf A}
\newcommand{\B}{\mathbf B}
\newcommand{\R}{\mathbb R}
\renewcommand{\S}{\mathfrak S}
\newcommand{\T}{\mathcal T}
\newcommand{\D}{\boldsymbol{\mathscr D}}
\newcommand{\I}{\boldsymbol{\mathscr I}}
\newcommand{\X}{\boldsymbol{\mathscr X}}
\newcommand{\Y}{\boldsymbol{\mathscr Y}}
\newcommand{\Z}{\boldsymbol{\mathscr Z}}

\begin{document}

\title{An Axiomatic Approach to Tensor Rank Functions}
\author{Wayne W. Wheeler}
\address{Center for Communications Research, 4320 Westerra Court, San Diego, CA 92121}
\email{wheeler@ccrwest.org}

\begin{abstract}
Recent work of Qi \etal~\cite{Qi20} proposes a set of axioms for tensor
rank functions. The current paper presents examples showing that their axioms
allow rank functions to have some undesirable properties, and a stronger set of
axioms is suggested that eliminates these properties. Two questions raised by
Qi \etal\ involving the submax rank function are also answered.
\end{abstract}

\maketitle

\section{Introduction}

Tensors are multidimensional arrays that provide a natural generalization
of matrices. The theory was originally developed in psychometrics in the work
of authors including Hitchcock~\cite{Hitchcock27a, Hitchcock27b},
Cattell~\cite{Cattell44}, Tucker~\cite{Tucker63, Tucker64, Tucker66}, Carroll
and Chang~\cite{Carroll70}, and Harshman~\cite{Harshman90}.  Tensors have
subsequently proven to be useful in numerous other applications such as
chemometrics~\cite{Appellof81}, signal processing~\cite{DeLathauwer98},
numerical analysis~\cite{Beylkin05, Ibraghimov02}, computer
vision~\cite{Vasilescu02}, neuroscience~\cite{Acar07a, Acar07b}, and graph
analytics~\cite{Bader07a, Kolda05}. The concept of the {\em canonical
polyadic rank\/} of a tensor, first proposed by
Hitchcock~\cite{Hitchcock27a, Hitchcock27b} in 1927, is of fundamental
importance since many applications involve approximating a tensor by another
tensor of low rank.

Recent work of Qi \etal~\cite{Qi20} uses an axiomatic approach to study a more
general notion of tensor rank. The authors propose a set of axioms for a tensor
rank function, define a partial order on the class of all such functions, and
show that there is a unique minimum rank function under this partial order.
They then consider some specific rank functions, one of which we call the
{\em submax rank.} They propose this function as a candidate for the minimum
rank function satisfying their axioms.

The current paper continues the axiomatic approach. After reviewing terminology
and fixing notation in Section~\ref{sec:background}, Section~\ref{sec:QZCrank}
studies the set of functions defined by the axioms of Qi \etal~\cite{Qi20},
which we call {\em QZC rank functions.} This section provides answers to two
questions about submax rank raised by Qi \etal\ In particular,
Proposition~\ref{prop:mincond} leads to an example showing that the submax
rank is not the minimum QZC rank function. The second question is related to
the property that any matrix of rank $R$ contains an $R\times R$ submatrix of
rank $R$. Qi \etal\ consider a similar but somewhat weaker property of some
QZC rank functions, and Corollary~\ref{cor:tucker} shows that the submax rank
does have this property.

Section~\ref{sec:QZCrank} also shows that QZC rank functions can have some
properties that seem quite undesirable. For example, it is possible for a QZC
rank function $r$ to satisfy $r(\D) < D$ when $\D$ is a diagonal tensor with
$D$ nonzero entries on the diagonal. In addition, Section~\ref{sec:QZCrank}
gives an example of a QZC rank function $r$ and two tensors $\X$ and $\Y$ such
that $\Y$ is obtained from $\X$ simply by appending a slab of zeros, but
$r(\Y) > r(\X)$. To eliminate these sorts of examples, a different set of
axioms for tensor rank functions is proposed in Section~\ref{sec:rank}. All
tensor rank functions satisfying these axioms are QZC rank functions, but they
do not have the same sort of pathological behavior.

\section{Background and Notation}
\label{sec:background}

This section describes the notation and terminology used in the remainder
of the paper. The notation is generally intended to conform to that used by
Kolda and Bader~\cite{Kolda09} or to Qi \etal~\cite{Qi20}. For simplicity all
tensors considered in this paper will have entries in $\R$.

The {\em order\/} of a tensor is the number of dimensions, which are also
called {\em modes\/} or {\em ways.} Vectors are simply tensors of order one
and are written as boldface lowercase letters such as $\mathbf a$; matrices
are tensors of order two and are written as boldface capital letters such as
$\A$; tensors of higher order or of unspecified order are written as boldface
Euler script letters such as $\X$. The $i^\text{th}$ entry of a vector
$\mathbf a$ is denoted by $a_i$, the $(i,j)$ entry of a matrix $\A$ is denoted
by $a_{ij}$, and the $(i_1,\dots,i_N)$ entry of a tensor $\X$ of order $N$
is denoted by $x_{i_1\cdots i_N}$.  It will be convenient to write $\T$ for the
collection of all tensors over $\R$. The notation $\S_N$ will denote the
symmetric group of all permutations of the set $\{1,\dots,N\}$.

When considering a tensor $\X$ of order $N$, we generally assume that the
$n^\text{th}$ index ranges from~$1$ to~$I_n$. In this case we write
$\X\in\R^{I_1\times\cdots\times I_N}$. If $N > 1$ and $I_n = 1$ for some $n$, then
$\X$ has an associated tensor of order $N-1$ obtained by eliminating the index
corresponding to the $n^\text{th}$ mode of $\X$.

Let $\X\in\R^{I_1\times\cdots\times I_N}$ be a tensor. For $1\leq n\leq N$ suppose
that $1\leq J_n\leq I_n$ and $1\leq i_{n1}<\cdots<i_{nJ_n}\leq I_n$.  Then the
tensor $\Y\in\R^{J_1\times\cdots\times J_N}$ given by 
\[ y_{j_1\cdots j_N} = x_{i_{1j_1}\cdots i_{Nj_N}} \]
is called a {\em subtensor\/} of $\X$. When $J_n = 1$ for all but one or two of
the dimensions $n$ with $1\leq n\leq N$, the subtensor $\Y$ can be identified
with either a vector $\mathbf y$ or a matrix $\mathbf Y$. In these cases it
will be convenient to refer to either $\mathbf y$ or $\mathbf Y$ as a
{\em submatrix\/} of $\Y$, even though $\Y$ may not itself be a matrix.

Let $\mathbf a^{(n)}\in\R^{I_n}$ be a nonzero vector for $1\leq n\leq N$. Let
$\X\in\R^{I_1\times\cdots\times I_N}$ be the tensor given by
\[ x_{i_1\cdots i_N} = a^{(1)}_{i_1}\cdots a^{(N)}_{i_N}. \]
If $\circ$ denotes the outer product of vectors, then this tensor is often
written as
\[ \X = \mathbf a^{(1)}\circ\cdots\circ\mathbf a^{(N)}. \]
A tensor of this form is said to have {\em rank one.}

The idea of expressing a tensor as the sum of a finite number of rank-one
tensors is originally due to Hitchcock~\cite{Hitchcock27a, Hitchcock27b}.
He proposed defining the rank of a tensor $\X$ to be the minimum number of
rank-one tensors having $\X$ as their sum. The notion was not widely studied,
however, until Kruskal~\cite{Kruskal77} proposed the definition independently
in 1977. This idea is now the most commonly used definition of tensor rank and
is often called the {\em canonical polyadic rank\/} or {\em CP rank.} It is
known that the CP rank of a tensor depends upon the base field. An
example of a class of tensors in $\R^{2\times 2\times 2}$ that have CP rank~3 as
real tensors but CP rank~2 as complex tensors appears in work of de~Silva and
Lim~\cite[Section~7.4]{DeSilva08}.

A tensor $\D\in\R^{I_1\times\cdots\times I_N}$ is said to be {\em diagonal\/}
if $i_1 = \cdots = i_N$ whenever $d_{i_1\cdots i_N}\neq 0$. The entries for which
$i_1 = \cdots = i_N$ form the {\em diagonal,} which some authors call the
{\em superdiagonal.} One important example of a diagonal tensor is the analogue
of the $M\times M$ identity matrix $\mathbf I_M$. The {\em identity tensor\/}
$\I_{M, N}$ is defined to be the $N$-dimensional diagonal tensor with
$I_1 = \cdots = I_N = M$ and ones along the diagonal. The notation $\I_{M,N}$
differs from that used by Qi \etal~\cite{Qi20}; in their notation the order of
the two subscripts is interchanged.

It is often useful to rearrange the entries of a tensor $\X$ into a matrix.
This process, known as {\em unfolding\/} or {\em matricization,\/} is discussed
in some detail by Kolda~\cite{Kolda06}. Only the special case of mode-$n$
unfolding will be needed in this paper. The mode-$n$ unfolding of a tensor
$\X\in\R^{I_1\times\cdots\times I_N}$ is a matrix
$\mathbf X_{(n)}\in\R^{I_n\times J_n}$, where
$J_n = I_1\cdots I_{n-1}I_{n+1}\cdots I_N$. The element in position
$(i_1,\dots,i_N)$ of the tensor $\X$ is mapped to the element in position
$(i_n, j)$ of the matrix $\mathbf X_{(n)}$, where
\[ j = 1 + \sum_{k\neq n}(i_k-1)J_k \]
and
\[ J_k = \prod_{\substack{1\leq m<k\\ m\neq n}}I_m. \]
The important fact about the mode-$n$ unfolding $\mathbf X_{(n)}$ is that
its columns are precisely the vectors obtained as subtensors of $\X$ by fixing
every coordinate except the one in mode~$n$. Some authors such as De~Lathauwer
\etal~\cite{DeLathauwer00} use a different ordering for the columns of
$\mathbf X_{(n)}$.

The {\em mode-$n$ matrix product\/} of a tensor
$\X\in\R^{I_1\times\cdots\times I_N}$ with a matrix $\A\in\R^{J\times I_n}$ is the
tensor $\X\times_n\mathbf A\in
\R^{I_1\times\cdots\times I_{n-1}\times J\times I_{n+1}\times\cdots\times I_N}$ given by
\[
(\X\times_n\mathbf A)_{i_1\cdots i_{n-1}ji_{n+1}\cdots i_N}
  = \sum_{i_n=1}^{I_n}x_{i_1\cdots i_N}a_{ji_n}.
\]
It is easy to see that if $m\neq n$, then
\[ \X\times_m\B\times_n\A = \X\times_n\A\times_m\B{\rm;} \]
if $m=n$, then
\[ \X\times_n\B\times_n\A = \X\times_n(\A\B). \]
Suppose that $\A^{(n)}\in\R^{J_n\times I_n}$ for $1\leq n\leq N$. Then the
{\em Tucker operator}~\cite{Kolda06} is defined by
\[
\llbracket\X;\A^{(1)},\dots,\A^{(N)}\rrbracket
  = \X\times_1\A^{(1)}\times_2\cdots\times_N\A^{(N)}.
\]
The following result gives a useful relationship between tensor unfoldings and
the Tucker operator.

\begin{proposition}[{\cite[Proposition~3.7(c)]{Kolda06}}]
\label{prop:unfold}
Suppose that $\X\in\R^{I_1\times\cdots\times I_N}$ and
$\A^{(n)}\in\R^{J_n\times I_n}$ for $1\leq n\leq N$. Then the following conditions
are equivalent:
\begin{enumerate}
\item $\Y = \llbracket\X;\A^{(1)},\dots,\A^{(N)}\rrbracket${\rm;}
\item $\mathbf Y_{(n)} = \A^{(n)}\mathbf X_{(n)}\bigl(\A^{(N)}\otimes\cdots\otimes
\A^{(n+1)}\otimes\A^{(n-1)}\otimes\cdots\otimes\A^{(1)}\bigr)^{\mathsf T}$ for every
$n$ with $1\leq n\leq N${\rm;}
\item $\mathbf Y_{(n)} = \A^{(n)}\mathbf X_{(n)}\bigl(\A^{(N)}\otimes\cdots\otimes
\A^{(n+1)}\otimes\A^{(n-1)}\otimes\cdots\otimes\A^{(1)}\bigr)^{\mathsf T}$ for some
$n$ with $1\leq n\leq N$.
\end{enumerate}
\end{proposition}

\section{QZC rank functions and the submax rank}
\label{sec:QZCrank}

This section is devoted to studying the class of rank functions defined by
the axioms of Qi \etal~\cite{Qi20}, which we call {\em QZC rank functions\/}
after the authors, Qi, Zhang, and Chen.

\begin{definition}
\rm A function $r:\T\to\mathbb N\cup\{0\}$ will be called a {\em QZC rank
function\/} if it satisfies the following axioms:
\begin{enumerate}
\item[\rm(QZC1)] $r(\X) = 0$ if and only if $\X = \mathbf 0$, and $r(\X) = 1$
if and only if $\X$ is a rank-one tensor.
\item[\rm(QZC2)] If $N\geq 2$, then $r(\I_{M,N}) = M$.
\item[\rm(QZC3)] If $\X\in\R^{I_1\times I_2\times1\times\cdots\times1}$, then $r(\X)$
is equal to the matrix rank of the $I_1\times I_2$ matrix corresponding to $\X$.
\item[\rm(QZC4)] $r(\X) = r(\alpha\X)$ for all $\alpha\in\R-\{0\}$.
\item[\rm(QZC5)] Suppose that $\X\in\R^{I_1\times\cdots\times I_N}$ and
$\pi\in\S_N$. Then the tensor $\Y$ given by
$y_{i_1\cdots i_N} = x_{i_{\pi(1)}\cdots i_{\pi(N)}}$ satisfies $r(\Y) = r(\X)$.
\item[\rm(QZC6)] If $\Y$ is a subtensor of $\X$, then $r(\Y)\leq r(\X)$.
\end{enumerate}
\end{definition}

\begin{proposition}[{\cite[Theorem~2.2]{Qi20}}]
Let $r_1$ and $r_2$ be QZC rank functions. Then the functions
$r,R:\T\to\mathbb N\cup\{0\}$ given by
\[ r(\X) = \min\{r_1(\X),r_2(\X)\} \]
and
\[ R(\X) = \max\{r_1(\X),r_2(\X)\} \]
are QZC rank functions.
\end{proposition}

Define a partial ordering $\preceq$ on the collection of all QZC rank
functions by setting $r_1\preceq r_2$ if and only if $r_1(\X)\leq r_2(\X)$ for
every tensor $\X$. The maximum of two QZC rank functions will not be used in
this work, but the minimum is interesting because of the following result.

\begin{proposition}[{\cite[Theorem~2.3]{Qi20}}]
\label{prop:min}
There is a unique minimum QZC rank function $\mu$ given by
\[ \mu(\X) = \min\{r(\X)\mid\text{\rm $r$ is a QZC rank function}\}. \]
\end{proposition}

Suppose that $\X\in\R^{I_1\times\cdots\times I_N}$. For $1\leq n\leq N$ let
$R_n$ denote the rank of the mode-$n$ unfolding $\mathbf X_{(n)}$ of $\X$. The
$N$-tuple $(R_1,\dots,R_N)$ is a special case of the multiplex rank introduced
by Hitchcock~\cite{Hitchcock27b}; it is sometimes called the {\em multilinear
rank}~\cite{DeSilva08} of $\X$. Qi \etal~\cite{Qi20} show that the function
\begin{equation}
\label{eq:max}
r(\X) = \max\{R_1,\dots,R_N\}
\end{equation}
is a QZC rank function. In addition, they define $\submax\{R_1,\dots,R_N\}$ to
be the second largest value of the multiset $\{R_1,\dots,R_N\}$ if $N > 1$ and
$\submax\{R_1\} = R_1$ if $N=1$; for example,
\[ \submax\{1, 2, 3, 3\} = 3. \]
They then show that the function
\begin{equation}
\label{eq:submax}
r(\X) = \submax\{R_1,\dots,R_N\}
\end{equation}
is also a QZC rank function. Qi \etal\ call the function defined by
Equation~(\ref{eq:max}) the {\em max Tucker rank\/} and the function defined
by Equation~(\ref{eq:submax}) the {\em submax Tucker rank}; for simplicity
we refer to them as the {\em max rank\/} and the {\em submax rank\/}.

If $\mathbf X\in\R^{I_1\times I_2}$ is a nonzero matrix of rank $R$, then
$\mathbf X$ has an $R\times R$ submatrix of rank $R$. Unfortunately, the
analogous property is not always satisfied for tensor rank functions. For
example, let $r$ be the max rank, and consider a tensor
$\X\in\R^{I_1\times I_2\times I_3}$ with $r(\X) = R$. It is quite easy to construct
examples in which $R > I_3$ so that $\X$ can have no subtensor
$\Y\in\R^{R\times R\times R}$, let alone one with $r(\Y) = R$. But some QZC rank
functions $r$ do have the weaker property that a nonzero tensor $\X$ always has
a subtensor $\Y\in\R^{J_1\times\cdots\times J_N}$ such that
\[ r(\X) = r(\Y) = J_n \]
for some $n$ with $1\leq n\leq N$. In fact, studying QZC rank functions with
this property is one of the main motivations for the work of Qi
\etal~\cite{Qi20}. They show that the max rank has this property and ask
whether the submax rank does. The following lemma will lead to an answer to
this question in Corollary~\ref{cor:tucker}.

\begin{lemma}
\label{lem:rank}
Let $\X\in\R^{I_1\times\cdots\times I_N}$ be a nonzero tensor, and set
$R = \rank\bigl(\mathbf X_{(n)}\bigr)$ for some $n$ with $1\leq n\leq N$. Let
$j_1 < \cdots < j_R$ be the indices of $R$ linearly independent rows of
$\mathbf X_{(n)}$, and define $\A\in\R^{R\times I_n}$ by
\[
a_{st} = \begin{cases}
        1 & \text{if $t=j_s$}, \\
        0 & \text{otherwise.}
        \end{cases}
\]
Set $\Y = \X\times_n\A$. Then $\Y$ is a subtensor of $\X$ with
$\rank\bigl(\mathbf Y_{(m)}\bigr) = \rank\bigl(\mathbf X_{(m)}\bigr)$ for
$1\leq m\leq N$.
\end{lemma}
\begin{proof}
By permuting the coordinates of $\X$ if necessary, we may assume without
loss of generality that $n=1$. It is easy to check that
$\Y\in\R^{R\times I_2\times\cdots\times I_N}$ is a subtensor of $\X$. Row~$s$ of the
matrix $\mathbf Y_{(1)} = \A\mathbf X_{(1)}$ is equal to row~$j_s$ of the matrix
$\mathbf X_{(1)}$ for $1\leq s\leq R$. Because
$\rank\bigl(\mathbf X_{(1)}\bigr) = R$ and rows $j_1,\dots,j_R$ of
$\mathbf X_{(1)}$ are linearly independent, $\mathbf X_{(1)}$ and
$\mathbf Y_{(1)}$ have the same rank. Moreover, there is a matrix
$\B\in\R^{I_1\times R}$ such that $\B\A\mathbf X_{(1)} = \mathbf X_{(1)}$,
so $\X = \X\times_1\B\A = \Y\times_1\B$. By Proposition~\ref{prop:unfold} it
follows that if $2\leq m\leq N$, then
\[
\mathbf X_{(m)}
  = \mathbf Y_{(m)}(\mathbf I_{I_N}\otimes\cdots\otimes\mathbf I_{I_{m+1}}\otimes
    \mathbf I_{I_{m-1}}\otimes\cdots\otimes\mathbf I_{I_2}\otimes\B^{\mathsf T})
\]
and
\[
\mathbf Y_{(m)}
  = \mathbf X_{(m)}(\mathbf I_{I_N}\otimes\cdots\otimes\mathbf I_{I_{m+1}}\otimes
    \mathbf I_{I_{m-1}}\otimes\cdots\otimes\mathbf I_{I_2}\otimes\A^{\mathsf T}).
\]
Thus $\mathbf X_{(m)}$ and $\mathbf Y_{(m)}$ have the same column
space, so they must have the same rank.
\end{proof}

\begin{proposition}
Let $\X\in\R^{I_1\times\cdots\times I_N}$ be a nonzero tensor, and set
$R_n = \rank\bigl(\mathbf X_{(n)}\bigr)$ for $1\leq n\leq N$. Then $\X$ has a
subtensor $\Y\in\R^{R_1\times\cdots\times R_N}$ with
$R_n = \rank\bigl(\mathbf Y_{(n)}\bigr)$ for $1\leq n\leq N$.
\end{proposition}
\begin{proof}
Lemma~\ref{lem:rank} shows that $\X$ has a subtensor
$\X'\in\R^{R_1\times I_2\times\cdots\times I_N}$ such that
$\rank\bigl(\mathbf X'_{(n)}\bigr) = R_n$ for all $n$ with $1\leq n\leq N$.
Applying the lemma inductively gives the desired result.
\end{proof}

\begin{corollary}
\label{cor:tucker}
Let $r$ denote either the max rank or the submax rank. If $\X$ is a nonzero
tensor, then there is a subtensor $\Y\in\R^{J_1\times\cdots\times J_N}$ such that
\[ r(\X) = r(\Y) = J_n \]
for some $n$ with $1\leq n\leq N$.
\end{corollary}

Qi \etal~\cite{Qi20} ask whether the minimum QZC rank function $\mu$ given by
Proposition~\ref{prop:min} is equal to the submax rank. To answer this
question, we begin by letting $\sim$ denote the weakest equivalence relation
satisfying the following conditions on the collection $\T$ of all real tensors:
\begin{enumerate}
\item If $\X\in\T$ and $\alpha\in\R-\{0\}$, then $\X\sim\alpha\X$.
\item Suppose that $\X\in\R^{I_1\times\cdots\times I_N}$ and $\pi\in\S_N$. Let
$\Y$ be the tensor given by $y_{i_1\cdots i_N} = x_{i_{\pi(1)}\cdots i_{\pi(N)}}$.
Then $\X\sim\Y$.
\end{enumerate}

\begin{proposition}
\label{prop:mincond}
Let $r:\T\to\mathbb N\cup\{0\}$ be a function satisfying the following
conditions:
\begin{enumerate}
\item $r(\X) = 0$ if and only if $\X = \mathbf 0$, and $r(\X) = 1$ if and only
if $\X$ is a rank-one tensor.
\item $r$ is constant on equivalence classes.
\item Suppose that $\X\neq\mathbf 0$ and $\X$ is not a rank-one tensor.
Set
\[ S_0 = \{M\mid\text{$\I_{M,N}$ is a subtensor of some $\Y\sim\X$}\} \]
and
\[
S_1 = \{\rank(\A)\mid\text{$\A$ is a submatrix of some $\Y\sim\X$}\}.
\]
The value $r(\X)$ is given by $r(\X) = \max(S_0\cup S_1\cup\{2\})$.
\end{enumerate}
Then $r = \mu$.
\end{proposition}
\begin{proof}
The first step is to show that $r$ is a QZC rank function. If $N\geq 2$ and
$M = 1$, then $\I_{M,N}$ is a rank-one tensor, so $r(\I_{M,N}) = M$ by
Condition~(1). If $M > 1$, then $\I_{M,N}$ is not a rank-one tensor. Suppose
that $\Y\sim\I_{M,N}$ and $\A$ is a submatrix of $\Y$. Then one can easily
check that $\rank(\A)\leq 1$. Thus Condition~(3) implies that $r(\I_{M,N}) = M$,
and Axiom~(QZC2) is satisfied.

Suppose that $\X\in\R^{I_1\times I_2\times 1\times\cdots\times1}$, that
$\X\neq\mathbf 0$, and that $\X$ is not a rank-one tensor. Axiom~(QZC3) is
clearly satisfied unless there is a $\Y\sim\X$ such that $\I_{M,N}$ is a
subtensor of $\Y$ with $M\geq 2$. But in this case $\Y$ and $\X$ must be
tensors of order $N$, so $N=2$ and $\X\in\R^{I_1\times I_2}$. Thus the tensors
$\X$ and $\Y$ are actually matrices $\mathbf X$ and $\mathbf Y$. The identity
matrix $\mathbf I_M$ is a submatrix of $\mathbf Y$, so
$\rank(\mathbf Y)\geq M$. It follows that
\[
r(\X) = \max\{\rank(\A)\mid\text{$\A$ is a submatrix of some $\Y\sim\X$}\}
      = \rank(\mathbf X).
\]
Thus Axiom~(QZC3) holds.

The condition that $r$ is constant on equivalence classes is equivalent to
Axioms~(QZC4) and~(QZC5), and Condition~(1) is simply a restatement
of Axiom~(QZC1). To prove that $r$ is a QZC rank function,
therefore, it only remains to prove that Axiom~(QZC6) is satisfied. Suppose
that $\Y$ is a subtensor of $\X$. To prove that $r(\Y)\leq r(\X)$, we may
assume that $\Y\neq\mathbf 0$ so that $\X\neq\mathbf 0$. If $r(\Y) = 1$, then
$r(\Y)\leq r(\X)$ by Condition~(1). Since any nonzero subtensor of a rank-one
tensor is itself a rank-one tensor, we may assume by Condition~(1) that
$\X\neq\mathbf 0$, $\Y\neq\mathbf 0$, and that neither $\X$ nor $\Y$ is a
rank-one tensor.

If $\Y'\sim\Y$, then one can easily check that there is a tensor
$\X'\sim\X$ such that $\Y'$ is a subtensor of $\X'$. Thus if $\I_{M,N}$ is a
subtensor of $\Y'$, then $\I_{M,N}$ is also a subtensor of $\X'$. Similarly,
if $\A$ is a submatrix of $\Y'$, then $\A$ is also a submatrix of $\X'$.
Condition~(3) now implies that $r(\Y)\leq r(\X)$, and $r$ is a QZC rank
function.

Every QZC rank function satisfies Conditions~(1) and~(2). Suppose that
$\X\neq\mathbf 0$ and $\X$ is not a rank-one tensor so that $\mu(\X)\geq 2$.
If $\Y\sim\X$ and $\I_{M,N}$ is a subtensor of $\Y$, then $\mu$ satisfies
$M = \mu(\I_{M,N})\leq\mu(\Y) = \mu(\X)$. Similarly, if $\A$ is a submatrix
of $\Y$, then Axioms~(QZC3),~(QZC5), and~(QZC6) imply that
$\rank(\A)\leq\mu(\Y) = \mu(\X)$. Thus $r(\X)\leq\mu(\X)$, and the
minimality of $\mu$ implies that $\mu = r$.
\end{proof}

We can now provide an example showing that the minimum QZC rank function
$\mu$ is not equal to the submax rank.

\begin{example}
\label{ex:diagonal}
\rm Let $\D$ denote the $3\times 3\times 3$ diagonal tensor with
$d_{111} = d_{222} = 1$ and $d_{333} = -1$, and let $\sigma$ denote the submax
rank. It is easy to check that the three unfoldings $\mathbf D_{(1)}$,
$\mathbf D_{(2)}$, and $\mathbf D_{(3)}$ all have rank~$3$, so $\sigma(\D) = 3$.
But if $\Y\sim\D$, then every nonzero submatrix of $\Y$ has rank one. Moreover,
exactly two of the nonzero entries of $\Y$ are equal, so the largest value of
$M$ for which $\I_{M,N}$ can be a subtensor of $\Y$ is $M = 2$. Thus
$\mu(\D) = 2$, and $\mu$ is not equal to $\sigma$.
\end{example}

The previous example shows that it is possible for a QZC rank function $r$
to have the property that $r(\D) < D$ even when $\D$ is a diagonal tensor with
$D$ nonzero entries along the diagonal. This property seems undesirable, but
the next result provides an example of a QZC rank function having a property
that seems even less desirable.

\begin{proposition}
\label{prop:add0}
Let $r:\T\to\mathbb N\cup\{0\}$ be a function satisfying the following
conditions:
\begin{enumerate}
\item $r(\X) = 0$ if and only if $\X = \boldsymbol 0$, and $r(\X) = 1$ if and
only if $\X$ is a rank-one tensor.
\item $r$ is constant on equivalence classes.
\item Suppose that $\X\in\R^{I_1\times\cdots\times I_N}$ and there are exactly
two modes $n_1 < n_2$ such that $I_{n_1} > 1$ and $I_{n_2} > 1$. If
$\mathbf X\in\R^{I_{n_1}\times I_{n_2}}$ is the matrix associated to $\X$, then
$r(\X) = \rank(\mathbf X)$.
\item Suppose that $\X\in\R^{I_1\times\cdots\times I_N}$ does not have rank one and
$\X\neq\mathbf 0$. If there are at least three modes $n$ with $I_n > 1$, then
$r(\X) = \max\{I_1,\dots,I_N\}$.
\end{enumerate}
Then $r$ is a QZC rank function.
\end{proposition}
\begin{proof}
It is easy to check that the four conditions given in the statement of the
proposition are consistent, so there is a unique function $r$ satisfying the
conditions. Axiom~(QZC1) is simply a restatement of Condition~(1), and
Axioms~(QZC4) and~(QZC5) are equivalent to Condition~(2). The function $r$
satisfies Axiom~(QZC3) by Condition~(3).

Assume that $N\geq 2$. If $M = 1$, then $\I_{M,N}$ has rank one, so
$r(\I_{M,N}) = 1$ by Condition~(1). If $M > 1$ and $N = 2$, then
$r(\I_{M,N}) = M$ by Condition~(3). Finally, if $M > 1$ and $N > 2$, then
Condition~(4) implies that $r(\I_{M,N}) = M$. Thus Axiom~(QZC2) is satisfied.

Finally, suppose that $\Y\in\R^{J_1\times\cdots\times J_N}$ is a subtensor of
$\X\in\R^{I_1\times\cdots\times I_N}$. If $\Y = \boldsymbol 0$, then
$r(\Y)\leq r(\X)$ by Condition~(1). If $\Y$ is a rank-one tensor, then
$\X\neq\boldsymbol 0$, so Condition~(1) implies that $r(\Y) = 1\leq r(\X)$. Now
suppose that $\Y\neq\boldsymbol 0$ does not have rank one. Then at least two
modes $n$ have dimension $J_n > 1$. If more than two modes have this property,
then Condition~(4) implies that
\[ r(\Y) = \max\{J_1,\dots,J_N\} \leq \max\{I_1,\dots,I_N\} = r(\X). \]
Thus we may assume that there are exactly two modes $n_1<n_2$ such that
$J_{n_1} > 1$ and $J_{n_2} > 1$. If $\mathbf Y\in\R^{J_{n_1}\times J_{n_2}}$ is the
matrix corresponding to $\Y$, then Condition~(3) implies that
$r(\Y) = \rank(\mathbf Y)\leq\max\{J_{n_1},J_{n_2}\}$. If exactly two modes of
$\X$ have dimension greater than one, then these modes must be $n_1$ and $n_2$.
Let $\mathbf X\in\R^{I_{n_1}\times I_{n_2}}$ be the matrix corresponding to $\X$.
Then $\mathbf Y$ is a submatrix of $\mathbf X$, and
\[ r(\Y) = \rank(\mathbf Y)\leq\rank(\mathbf X) = r(\X) \]
by Condition~(3). If more than two modes of $\X$ have dimension greater than
one, then
\[
r(\Y) = \rank(\mathbf Y)
      \leq\max\{J_{n_1},J_{n_2}\}
      \leq\max\{I_1,\dots,I_N\}
      = r(\X)
\]
by Condition~(4). Thus Axiom~(QZC6) is satisfied, and $r$ is a QZC rank
function.
\end{proof}

\begin{example}
\label{ex:slab}
\rm Let $r$ be the QZC rank function defined by the conditions given in
Proposition~\ref{prop:add0}. Consider the tensor $\X\in\R^{2\times 2\times 3}$
with
\[
x_{i_1i_2i_3} = \begin{cases}
              1 & \text{if $i_1=i_2=i_3=1$ or $i_1=i_2=i_3=2$,} \\
              0 & \text{otherwise.}
              \end{cases}
\]
Then Condition~(4) implies that $r(\X) = 3$. In particular,
$r(\X) > r(\I_{2,3})$, even though $\X$ is obtained by appending a slab of zeros
to the tensor $\I_{2,3}$.
\end{example}

The previous example gives a second undesirable property that a QZC rank
function may have. The next section gives a somewhat different set of axioms
that eliminates both of these properties.

\section{Axioms for tensor rank functions}
\label{sec:rank}

The results of Section~\ref{sec:QZCrank} show that the axioms for QZC rank
functions have at least two undesirable consequences: a diagonal tensor can
have a rank that is smaller than the number of nonzero diagonal entries, and
the rank of a tensor may increase when a slab of zeros is appended to it. In
this section the axioms of Qi \etal~\cite{Qi20} are modified to obtain a more
restrictive notion of tensor rank. The axioms proposed here rectify both of the
issues discussed in Section~\ref{sec:QZCrank}, but it is still possible that
they allow for other undesirable properties. Further modifications may be
necessary. 

\begin{definition}
\rm A {\em tensor rank function\/} is a function $r:\T\to\mathbb N\cup\{0\}$
satisfying the following axioms:
\begin{enumerate}
\item[\rm(TR1)] $r(\X) = 1$ if and only if $\X$ is a rank-one tensor.
\item[\rm(TR2)] If $N\geq 2$, then $r(\I_{M,N}) = M$.
\item[\rm(TR3)] If $\X\in\R^{I_1\times I_2}$, then $r(\X)$ is equal to the matrix
rank of $\X$.
\item[\rm(TR4)] If $\X\in\R^{I_1\times\cdots\times I_N}$ and $\X'$ is the
corresponding $(N+1)$-way tensor in $\R^{I_1\times\cdots\times I_N\times 1}$,
then $r(\X) = r(\X')$.
\item[\rm(TR5)] Suppose that $\X\in\R^{I_1\times\cdots\times I_N}$ and $\pi\in\S_N$.
Then the tensor $\Y$ given by $y_{i_1\cdots i_n} = x_{i_{\pi(1)}\cdots i_{\pi(N)}}$
satisfies $r(\Y) = r(\X)$.
\item[\rm(TR6)] If $\X\in\R^{I_1\times\cdots\times I_N}$ and $\A\in\R^{J\times I_n}$,
then $r(\X\times_n\A)\leq r(\X)$.
\end{enumerate}
\end{definition}

It is easy to see that the definition of CP rank proposed by
Hitchcock~\cite{Hitchcock27a, Hitchcock27b} and Kruskal~\cite{Kruskal77}
satifies the first five axioms, but it also satisfies Axiom~(TR6). Indeed,
suppose that $\X\in\R^{I_1\times\cdots\times I_N}$ is a nonzero tensor. Let
$\rank(\X)$ denote the smallest natural number $R$ such that $\X$ can be
expressed as a sum of $R$ tensors of rank one. Then there are nonzero vectors
$\mathbf a_r^{(n)}\in\R^{I_n}$ for $1\leq r\leq R$ and $1\leq n\leq N$ such that
\[ \X = \sum_{r=1}^R\mathbf a_r^{(1)}\circ\cdots\circ\mathbf a_r^{(N)}. \]
If $\A\in\R^{J\times I_n}$, then
\[
\X\times_n\A
  = \sum_{r=1}^R\mathbf a_r^{(1)}\circ\cdots\circ\mathbf a_r^{(n-1)}\circ
    \A\mathbf a_r^{(n)}\circ\mathbf a_r^{(n+1)}\circ\cdots\circ\mathbf a_r^{(N)},
\]
so $\rank(\X\times_n\A)\leq\rank(\X)$.

The following proposition is the first step toward showing that any tensor
rank function is also a QZC rank function.

\begin{proposition}
\label{prop:subtensor}
Let $r$ be a tensor rank function. If $\Y$ is a subtensor of $\X$, then
$r(\Y)\leq r(\X)$.
\end{proposition}
\begin{proof}
Suppose that $\Y\in\R^{J_1\times\cdots J_N}$ is a subtensor of
$\X\in\R^{I_1\times\cdots\times I_N}$. For $1\leq n\leq N$ there are indices
$i_{n1},\dots,i_{nJ_n}$ such that $1\leq i_{n1}<\cdots<i_{nJ_n}\leq I_n$ and
\[ y_{j_1\cdots j_N} = x_{i_{1j_1}\cdots i_{Nj_N}}. \]
Let $\A^{(n)}$ be the $J_n\times I_n$ matrix with a~$1$ in position $(j,i_{nj})$
for $1\leq j\leq J_n$ and zeros elsewhere. Then
$\Y = \llbracket\X;\A^{(1)},\dots,\A^{(N)}\rrbracket$, so Axiom~(TR6) implies that
$r(\Y)\leq r(\X)$.
\end{proof}

Example~\ref{ex:diagonal} describes a QZC rank function $r$ and a diagonal
tensor $\D$ with $D$ nonzero entries such that $r(\D) < D$. Such a phenomenon
cannot occur when $r$ is a tensor rank function. Indeed, if $D = 1$, then $\D$
is a rank-one tensor, and $r(\D) = D$ by Axiom~(TR1). If $D > 1$, then there is
an invertible diagonal matrix $\A$ such that $\D\times_1\A$ is a diagonal
tensor with $D$ ones on the diagonal. If $\D$ is a tensor of order $N$, then
$\I_{D,N}$ is a subtensor of $\D\times_1\A$, and $r(\I_{D,N}) = D$ by
Axiom~(TR2). Because $\A$ is invertible, Axiom~(TR6) implies that
$r(\D\times_1\A) = r(\D)$, and Proposition~\ref{prop:subtensor} shows that
\[ D = r(\I_{D,N})\leq r(\D\times_1\A) = r(\D). \]

\begin{proposition}
\label{prop:zero}
Let $r$ be a tensor rank function. Then $r(\X) = 0$ if and only if
$\X = \mathbf 0$.
\end{proposition}
\begin{proof}
Let $\Z\in\R^{I_1\times\cdots\times I_N}$ be the zero tensor, and let
$\mathbf a^{(n)}\in\R^{I_n}$ be any nonzero vector for $1\leq n\leq N$. Then
$\X = \mathbf a^{(1)}\circ\cdots\circ\mathbf a^{(N)}$ is a rank-one tensor, so
$r(\X) = 1$ by Axiom~(TR1). Let $\mathbf 0$ denote the $I_1\times I_1$ zero
matrix. Set
\[ \A = \begin{bmatrix} \mathbf I_{I_1} \\ \mathbf 0\end{bmatrix} \]
and $\Y = \X\times_1\A$. Then the zero tensor $\Z$ is a subtensor of
$\Y\in\R^{2I_1\times I_2\times\cdots\times I_N}$. Thus
Proposition~\ref{prop:subtensor} and Axiom~(TR6) imply that
$r(\Z)\leq r(\Y)\leq r(\X) = 1$. Since $\Z$ is not a rank-one tensor, it
follows that $r(\Z) = 0$ by Axiom~(TR1).

Conversely, suppose that $\X\in\R^{I_1\times\cdots\times I_N}$ is nonzero. Then
there are indices $i_1,\dots,i_N$ such that $x_{i_1\cdots i_N}\neq 0$, and $\X$
has a rank-one subtensor $\Y\in\R^{1\times\cdots\times 1}$ given by
$y_{1\cdots 1} = x_{i_1\cdots i_N}$. Thus $1 = r(\Y)\leq r(\X)$ by Axiom~(TR1) and
Proposition~\ref{prop:subtensor}, so $r(\X)\neq 0$.
\end{proof}

Example~\ref{ex:slab} gives a QZC rank function $r$ and two tensors $\X$ and
$\Y$ such that $r(\X) < r(\Y)$ even though $\Y$ can be obtained from $\X$
simply by adding a slab of zeros. A slight variant of the idea used to prove
Proposition~\ref{prop:zero} shows that this phenomenon cannot occur for tensor
rank functions. Indeed, suppose that $\X\in\R^{I_1\times\cdots\times I_N}$ and
\[ \A = \begin{bmatrix} \mathbf I_{I_n} \\ \mathbf 0\end{bmatrix}, \]
where $\mathbf 0$ denotes the $J_n\times I_n$ zero matrix for some $J_n > 0$.
Then $\Y = \X\times_n\A$ is obtained by appending $J_n$ slabs of zeros to $\X$
in mode $n$. But $\X = \Y\times_n\A^{\mathsf T}$ because $\A^{\mathsf T}$ is a left
inverse of $\A$, and it follows from Axiom~(TR6) that $r(\X) = r(\Y)$ for any
tensor rank function $r$.

Propositions~\ref{prop:subtensor} and~\ref{prop:zero} imply that any tensor
rank function is a QZC rank function. Many of the results for QZC rank
functions generalize in a straightforward way to tensor rank functions. For
example, it is easy to see that the class of all tensor rank functions contains
a unique minimum function. One could also define an equivalence relation
analogous to $\sim$ and use it to prove a result for tensor rank functions
similar to Proposition~\ref{prop:mincond}. We will not pursue these ideas
further.

The next result shows that tensor rank functions have another property
that one would expect.

\begin{proposition}
Suppose that $\X\in\R^{I_1\times\cdots\times I_N}$ is a tensor and
$\pi_n\in\S_{I_n}$ is a permutation for $1\leq n\leq N$. Let $r$ be a tensor
rank function. Then the tensor $\Y\in\R^{I_1\times\cdots\times I_N}$ given by
$y_{i_1\cdots i_N} = x_{\pi_1(i_1)\cdots\pi_N(i_N)}$ satisfies $r(\Y) = r(\X)$.
\end{proposition}
\begin{proof}
For $1\leq n\leq N$ let $\mathbf P^{(n)}\in\R^{I_n\times I_n}$ be the
permutation matrix corresponding to the permutation $\pi_n$. Then
$\Y = \llbracket\X;\mathbf P^{(1)},\dots,\mathbf P^{(N)}\rrbracket$, so
Axiom~(TR6) implies that $r(\Y)\leq r(\X)$. But each matrix $\mathbf P^{(n)}$
is invertible, so it follows easily that $r(\Y) = r(\X)$.
\end{proof}

It is interesting to consider the axioms for tensor rank functions in
the context of the higher-order singular value decomposition (HOSVD) developed
by De Lathauwer \etal~\cite[Theorem~2]{DeLathauwer00}. Their work shows that
any tensor $\X\in\R^{I_1\times\cdots\times I_N}$ can be written as
\[ \X = \llbracket\Y;\A^{(1)},\dots,\A^{(N)}\rrbracket, \]
where $\A^{(n)}\in\R^{I_n\times I_n}$ is an orthogonal matrix for $1\leq n\leq N$
and $\Y\in\R^{I_1\times\cdots\times I_N}$ is a {\em core tensor\/} satisfying
certain orthogonality and ordering properties. Because each $\A^{(n)}$ is
invertible, Axiom~(TR6) implies that $r(\X) = r(\Y)$ for any tensor rank
function $r$. Thus Axiom~(TR6) implies that tensor rank functions are constant
on all tensors having an HOSVD with the same core. Axiom~(TR2) specifies the
value of the rank when the core tensor is the identity, and Axiom~(TR3)
specifies its value on tensors that can be identified with matrices. But the
HOSVD allows for a wide variety of core tensors, so these conditions are not
actually very restrictive. It is quite possible that interesting subclasses of
tensor rank functions arise by restricting the values of the functions on
specific types of core tensors.

\bibliography{tensors}
\bibliographystyle{amsplain}

\end{document}